\documentclass[12pt]{article}

\usepackage{amssymb,url,color}
\usepackage{hyperref}
\newtheorem{theorem}{Theorem}[section]
\newtheorem{prop}[theorem]{Proposition}
\newtheorem{lemma}[theorem]{Lemma}
\newtheorem{cor}[theorem]{Corollary}
\newtheorem{unnumber}{}

\newtheorem{prepf}[unnumber]{Proof.}
\newenvironment{proof}{\prepf\rm}{\endprepf}
\newcommand{\qed}{\hfill$\Box$}

\newtheorem{prerk}[theorem]{Remark}
\newenvironment{remark}{\prerk\rm}{\endprerk}
\newtheorem{preprob}{Problem}[section]
\newenvironment{problem}{\preprob\rm}{\endpreprob}

\newcommand{\Aut}{\mathop{\mathrm{Aut}}}
\newcommand{\Inn}{\mathop{\mathrm{Inn}}}
\newcommand{\Out}{\mathop{\mathrm{Out}}}
\newcommand{\GF}{\mathop{\mathrm{GF}}}
\newcommand{\GL}{\mathop{\mathrm{GL}}}

\begin{document}

\title{Integrals of groups}
\author{Jo\~ao Ara\'ujo\footnote{Universidade Aberta, 
R. Escola Politecnica 147, 1269-001, Lisboa, Portugal and
CEMAT-Ci\^encias, Faculdade de Ci\^encias, Universidade de Lisboa, 1749-016, Portugal
\texttt{\href{mailto:joao.araujo@uab.pt}{joao.araujo@uab.pt}}}, 
Peter J. Cameron\footnote{School of Mathematics and Statistics, University
of St.~Andrews, UK and CEMAT-Ci\^encias, Faculdade de Ci\^encias, Universidade de Lisboa, 1749-016, Portugal, \texttt{\href{mailto:pjc20@st-andrews.ac.uk}{pjc20@st-andrews.ac.uk}}},
Carlo Casolo\footnote{Dipartimento di Matematica e Informatica ``U. Dini'', Universit\`a  di Firenze, Italia, \texttt{\href{mailto:carlo.casolo@unifi.it}{carlo.casolo@unifi.it}}},
and 
Francesco Matucci\footnote{Instituto de Matem\'{a}tica, Estat\'{i}stica e Computa\c c\~{a}o Cient\'{i}fica, Universidade Estadual de Campinas (UNICAMP), S\~{a}o Paulo, Brasil, \texttt{
\href{mailto:francesco@ime.unicamp.br}{francesco@ime.unicamp.br}}}}
\date{}
\maketitle

\begin{abstract}
An \emph{integral} of a group $G$ is a group $H$ whose derived group
(commutator subgroup) is isomorphic to $G$. This paper discusses integrals of
groups, and in particular questions about which groups have integrals and how
big or small those integrals can be. Our main results are:
\begin{itemize}\itemsep0pt
\item If a finite group has an integral, then it has a finite integral.
\item A precise characterization of the set of natural numbers $n$ for which
every group of order $n$ is integrable: these are the cubefree numbers $n$
which do not have prime divisors $p$ and $q$ with $q\mid p-1$.
\item An abelian group of order $n$ has an integral of order at most
$n^{1+o(1)}$, but may fail to have an integral of order bounded
by $cn$ for constant $c$.
\item A finite group can be integrated $n$ times (in the class of finite groups)
if and only if it is the central product of an abelian group and a perfect
group.
\end{itemize}
There are many other results on such topics as centreless groups, groups with
composition length~$2$, and infinite groups.
We also include a number of open problems.
\end{abstract}

\section{Introduction}

Let $G$ be a group. An \emph{integral} (in the sense of \emph{antiderivative})
of $G$ is any group $H$ such that $H' = G$.
Several authors refer to groups with an integral as \emph{$C$-groups}
or \emph{commutator-realizable groups}.

Questions about integrals of groups were first raised by Bernhard Neumann
in~\cite{neumann1}. It is surprising how little progress has been achieved
on this topic since Neumann's paper. Perhaps the explanation is that the
proofs rely as much on intricate constructions as on long arguments (though
the proofs of Theorems \ref{t:abelian} and \ref{t:orders} are substantial).

Some groups have no integral, as observed in Group Properties \cite{not-commutator-realizable}:

\begin{quote}
It is not true that every group can be realized as the derived subgroup
of another group -- for instance, the ``characteristically metacyclic and
commutator-realizable implies abelian'' statement tells us that a group whose
first two abelianizations are cyclic, but whose second derived subgroup is not
trivial, cannot arise as a derived subgroup.
\end{quote}

Every abelian group is integrable; so the set of orders of non-integrable
groups is a subset of the set of orders of non-abelian groups. A major
result, which occupies Section~\ref{s:orders}, is an exact description of
this set, more conveniently expressed in terms of its complement: we show that
every group of order $n$ is integrable if and only if $n$ is cube-free and
does not have prime divisors $p$ and $q$ such that $q\mid p-1$. This implies,
in particular, that for every even integer greater than $4$ there is a
non-integrable group of order $n$.

We will see in Section \ref{s:prelim} that if a finite group has an integral,
then it has a finite integral; so it makes sense to ask for a good upper bound
for the smallest integral of a finite integrable group. In order to have a good
computational test for integrability, it is useful to have such a bound. 
We conjecture that an integrable group of order $n$ has an integral of order
at most $n^3$; this is best possible, as shown by the cyclic group of
order~$2$, whose smallest integrals are the dihedral and quaternion groups of
order~$8$. We have been unable to find a closed formula for the smallest
integral of an abelian group, but we give a number of constructions in
Section~\ref{s:abel} for small integrals of abelian groups. From these
constructions it follows that an abelian group of order $n$ has an integral
of order $n^{1+o(1)}$. In the other direction, we show that
integrals of abelian groups of order~$n$ do not have order bounded by $cn$ for
any constant $c$. We also give a weaker bound $n^{c\log n}$ for the order
of the smallest integral of a centreless group of order $n$.

Also, there are infinite abelian groups $A$ which do not have
an integral $G$ with $|G:A|$ finite.

Any perfect group, and in particular any non-abelian simple group, is its own
integral.  Motivated by this, we analyse non-perfect groups with composition
length~$2$, and decide whether or not such groups are integrable (up to a
specific question about outer automorphism groups of simple groups). This
question leads to a more general topic: When does a group $G$ have an integral
inside some ``universal'' group $U$ containing $G$?

By analogy with $C^\infty$ functions in analysis, we are led to the question:
does there exist a  group which can be integrated infinitely often? Clearly
any perfect group has this property, since it is its own integral. 
Bernhard Neumann~\cite{neumann1} showed that there is no such
sequence where all the groups are finite and the sequence increases strictly;
but we give in Subsection~\ref{ss:infint} two sequences satisfying slightly
weaker conditions. We also examine finite groups which can be integrated $n$
times for every positive integer $n$. These turn out to be central products
of an abelian group and a perfect group.

We hope that the results about integrals will inspire the research on similar inverse problems in group theory, some of which will be discussed in a following
paper. 

\medskip

We are grateful to Alireza Abdollahi, Lars Jaffke, Michael Kinyon, and Avinoam Mann for
valuable comments. In particular, Alireza Abdollahi informed us of his paper
\cite{abdollahi}; although we have not seen the paper, the author kindly
communicated to us the main results. He also drew our attention to a comment
he had made on MathOverflow~\cite{overflow-integral}, and to the paper by
Filom and Miraftab~\cite{fm}. This paper was published in 2017, but the work
of these authors was completely independent of ours. Finally, Abdollahi is
(as far as we know) the first author to use the term ``integral'' in the
sense used here.

\section{Preliminaries}
\label{s:prelim}

In this section we gather some straightforward observations and prove that integrable finite groups have finite integrals. First we note
that, if a group has an integral, then it has infinitely many:

\begin{lemma}
Let $G$ be a group, let $H$ be an integral for $G$ and let $A$ be an abelian
group. Then $H\times A$ is an integral for $G$.
\label{l:const}
\end{lemma}

This is self-evident and already known (but we do not know where it was first
observed). It is the analogue of adding a ``constant of integration''.

\medskip

\begin{theorem}
Let $G$ be a finite group. If $G$ has an integral, then it has an integral
which is a finite group.
\label{t:finite}
\end{theorem}

\begin{proof}
Let $H$ be an integral of $G$. First, we reduce to the case where $H$ is
finitely generated.

Since $G=H'$, there are finitely many commutators $[h_1,k_1]$, \dots,
$[h_r,k_r]$ which generate $G$, with $h_i,k_i\in H$. Now it is clear that
the subgroup of $H$ generated by $h_1,\ldots,h_r,k_1,\ldots,k_r$ is an
integral of $G$, and is finitely generated. So, without loss, $H$ is finitely
generated.

Any conjugacy class in $H$ is contained in a coset of $H'=G$. (For 
$h,x\in H$, we have $Gx^{-1}hxh^{-1}=G$, so $Gx^{-1}hx=Gh$.) Thus $H$ is a
BFC-group (the conjugacy classes are finite of bounded order). While not
every BFC-group has centre of finite index, this is true for finitely 
generated BFC-groups. For let $H=\langle x_1,\ldots,x_k\rangle$. By 
assumption, $C_H(x_i)$ has finite index in $H$ (at most $|G|$); so
$Z(H)=\bigcap_{i=1}^kC_H(x_i)$ has finite index in $H$.

In consequence, $Z(H)$ is finitely generated abelian. So it has the form
$A\times B$, where $A$ is finite and $B$ is finitely generated torsion-free.
Since $B \le Z(H)$, it is a normal subgroup of $H$; since it is torsion-free,
$B\cap G=1$. Thus $G$ embeds in the finite group $\overline{H}=H/B$, and
$(\overline{H})'=G$.\qed
\end{proof}

\begin{remark}
In the paper~\cite{fm}, this theorem is proved under stronger hypotheses. In
addition, there are results in the paper which can be improved using our
theorem. For example, Theorem 20 asserts that, if a non-abelian $2$-group $G$ 
has cyclic centre and automorphism group a $2$-group, then any integral of $G$
is infinite; we can now conclude that such a group is not integrable.
\end{remark}

It is possible for a group to have infinitely many integrals with no abelian
direct factors
(see \cite{nm-extra}).
For example, every extraspecial $p$-group is an
integral of the cyclic group of prime order $p$. (A $p$-group $P$ is
\emph{special} if $P'=\Phi(P)=Z(P)$, and is \emph{extraspecial} if this
subgroup is cyclic of prime order.)

\begin{prop}
There exists a function $f$ defined on the natural numbers such that, if
$G$ is an integrable group of order $n$, then $G$ has an integral
of order at most $f(n)$.
\end{prop}

\begin{proof}
We can take $f(n)$ to be the maximum, over all integrable groups $G$ of 
order~$n$, of the minimum order of an integral of~$G$. (We will see at the
end of this section that, for every $n$, there is an integrable group of
order~$n$, so this function is well-defined.)
\qed\end{proof}

An algorithm for deciding whether a group $G$ of order $n$ is integrable could
run as follows: examine all groups $H$ of order at most $f(n)$ (and divisible
by $n$); return true if such an $H$ is found with $H'=G$, and false otherwise.
However, in the absence of a decent bound for $f(n)$, this algorithm is
worthless. It would be useful to have a good bound, and we pose the question
whether $f(n)<n^3$ for all natural numbers $n$.
(Note that the smallest integrals of the cyclic group of order~$2$ are the
dihedral and quaternion groups of order~$8$.)

In 
Section \ref{s:gs} 
we will prove a weaker bound for centreless groups.

\medskip

Finally, we consider abelian groups. Guralnick~\cite{gural1}
showed that, if $A$ is an abelian group of order $n$, then the group
$A\wr C_2$, of order $2n^2$, is an integral of $A$. In Section~\ref{s:abel}
we will find much smaller integrals of abelian groups, and pose the question
of finding the smallest.

\section{Some examples of integrability}
\label{s:exs}

In this section we give some examples of integrable and non-integrable groups.
These are of some interest in their own right, and will also be used in the
discussion of orders of non-integrable groups. We also give some results
concerning direct products.
We begin by recalling known results in the literature:
\begin{enumerate}
\item Abelian groups are integrable \cite{gural1}.
\item Dihedral groups are non-integrable. (This is stated by
Neumann~\cite{neumann1} with a reference to Zassenhaus~\cite{zassenhaus},
which we have not been able to check; an explicit proof, also showing that
quasi-dihedral and generalized quaternion groups are non-integrable, is in \cite[Corollary 18]{fm}.)
\item Symmetric groups $S_n$ (for $n \ge 3$) are non-integrable. (This is folklore. The first reference we found in print is in \cite{neumann1}, where it is
stated without proof; the earliest proof we found in print is \cite[Corollary 15 and Theorem 16]{fm}.)
\item Some matrix groups are integrable \cite{miller1}. Miller deals with all normal subgroups of the general linear groups $\mathrm{GL}_n(K)$, the unitary groups $\mathrm{U}_n(K)$, and the orthogonal groups $\mathrm{O}_n(K)$ for $K$ a field of characteristic different from $2$.
\end{enumerate}
We will recover some of the results above as consequences of our own constructions.

The next result is essentially the same as \cite[Theorem 17]{fm}.

\begin{prop}
Let $G$ be a group with a characteristic cyclic subgroup $C$ which is not
contained in $Z(G)$. Then $G$ has no integral.
\label{p:cc}
\end{prop}

\begin{proof}
Suppose for a contradiction that $H'=G$. Since $C$ is characteristic in $G$,
it is normal in $H$, and $H$ (acting by conjugation) induces a group of
automorphisms of $C$. The automorphism group of a cyclic group is abelian,
and so $G=H'$ acts trivially on $C$, and $C\le Z(G)$, a contradiction.\qed
\end{proof}

\begin{cor}
\label{thm:dihedral-no-integral}
If $n$ is even and $n>4$, the dihedral group of order $n$ is non-integrable.
\label{c:dihedral}
\end{cor}

\begin{proof}
The cyclic subgroup of order $n/2$ is characteristic (since all elements
outside it have order~$2$) and non-central (since $n>4$).\qed
\end{proof}

If $n=pq$, where $p$ and $q$ are primes, then a non-abelian group of order
$n$ exists if and only if $q\mid p-1$.

\begin{cor}
Let $p$ and $q$ be primes and $q\mid p-1$. Then the non-abelian group of
order $pq$ is centreless and non-integrable.
\label{c:pq}
\end{cor}

There are two non-abelian groups of order $p^3$. For $p=2$, these are the
quaternion group (which has an integral, namely $\mathrm{SL}(2,3)$), and the
dihedral group (which does not have an integral).

The pattern is similar for odd $p$. We prove the following:

\begin{theorem}
Of the two non-abelian groups of order $p^3$, where $p$ is an odd prime, the
group of exponent $p$ has an integral, while the group of exponent $p^2$
does not.
\label{t:primecubed}
\end{theorem}

\begin{proof}
The group of exponent $p$ is isomorphic to the group $G$ of upper unitriangular
$3\times3$ matrices over the field of order $p$. A short calculation with
matrices shows that $G=H'$, where $H$ is the group of upper triangular
matrices with non-zero elements on the diagonal.

The group of exponent $p^2$ has the presentation
\[G=\langle a,b:a^{p^2}=b^p=1,b^{-1}ab=a^{p+1}.\rangle\]
We first develop some properties of this group. Note that its centre has
order $p$ and is generated by $a^p$. Let $B=\langle a^p,b\rangle$. Then $B$
is elementary abelian of order $p^2$.

We show first that every element outside $B$ has order $p^2$. To see this,
note that $b^jab^{-j}=a^{-jp+1}$, so
\[b^ja^ib^{-j}=(b^jab^{-j})^i=a^{i(-jp+1)}.\]
Then we get
\begin{eqnarray*}
(a^ib^j)^p &=& a^i\cdot b^ja^ib^{-j}\cdot b^{2j}a^ib^{-2j}\cdots
b^{(p-1)j}a^ib^{-(p-1)j} \\
&=& a^i\cdot a^{i(-jp+1)}\cdot a^{i(-2jp+1)}\cdots a^{i(-(p-1)jp+1)} \\
&=& a^{pi},
\end{eqnarray*}
since $j(1+2+\cdots+(p-1))$ is divisible by $p$ (since $p$ is odd) and
$a^{p^2}=1$. Thus, $a^ib^j$ has order $p$ if and only if $p$ divides $i$,
which means that $a^ib^j\in B$.

Let $\alpha$ be an automorphism of $G$, and suppose that
$a^\alpha=a^ib^j$ (where $p$ does not divide $i$) and $b^\alpha=a^{pk}b^l$
(for if $b^\alpha \not \in B$, then $b^\alpha$ would have order $p^2$, which is impossible).
We must have
\[b^{-\alpha}a^\alpha b^\alpha=(a^\alpha)^{p+1}.\]
Since $a^{kp}$ is central, for the left-hand side we have
\begin{eqnarray*}
b^{-l}(a^ib^j)b^l &=& b^{-l}a^ib^lb^j \\
&=& a^{i(lp+1)}b^j.
\end{eqnarray*}
On the right we have
\[(a^ib^j)^{p+1} = (a^ib^j)^p a^i b^j = a^{i(p+1)}b^j.\]
So we must have $l=1$. There are $p(p-1)$ choices for $i$, $p$ for $j$, and
$p$ for $k$; so $|\Aut(G)|=p^3(p-1)$. The inner automorphism group has order
$|G/Z(G)|=p^2$, and so the outer automorphism group has order $p(p-1)$.

In more detail: Conjugation by $b$ corresponds to $(i,j,k)=(p+1,0,0)$, while
conjugation by $a$ corresponds to $(i,j,k)=(1,0,-1)$. So we can represent
the outer automorphism group by pairs $(i\pmod{p},j)$. Calculation shows that
this group is isomorphic to the $1$-dimensional affine group.

Now suppose that $H$ is an integral of $G$. Then $H$ acts on $G$ by 
conjugation, so there is a homomorphism $\theta:H\to\Aut(G)$, whose restriction
to $G$ maps $G$ to $\Inn(G)$. Write $\overline{H}$ and $\overline{G}$ for the images of
$H$ and $G$ under $\theta$. Then $(\overline{H})'=\overline{G}$, hence $\overline{H}/\overline{G}$ is
abelian, so either its order is $p$ or it divides $p-1$.

If $\overline{H}/\overline{G}$ has order $p$, then $|\overline{H}|=p^3$, and so it is not
possible that $|(\overline{H})'|=|\overline{G}|=p^2$.

Suppose that $m=|\overline{H}/\overline{G}|$ divides $p-1$. Up to conjugation, we may
assume that an element $\overline{h}$ of $\overline{H}$ of order $m$ is represented as an
outer automorphism of $G$ by a map with $j=0$. This means that the automorphism
fixes $b$. Its action on the quotient $G/Z(G)$, regarded as a $2$-dimensional
vector space, is a diagonal matrix with eigenvalues $\lambda$ and $1$. An
eigenspace with eigenvalue $\lambda$ has the property that all its cosets are
fixed. This means that all the automorphisms in $\overline{H}$ fix every coset of
a subgroup $K$ of $G$ of order $p^2$; so the commutator of any two of them
belongs to $K$. So the derived group of $H$ is contained in $K$, and cannot
be $G$.\qed
\end{proof}

For higher powers of a prime, a similar result holds:

\begin{prop}
Let $p$ be an odd prime and $n>3$. Let
\[G=\langle a,b\mid a^{p^{n-2}}=1, b^{p^2}=1, b^{-1}ab=a^{p^{n-3}+1}\rangle.\]
Then $G$ is not integrable.
\label{p:primepower}
\end{prop}

\begin{proof}
We first deal with the case $n=4$, following the arguments in the proof of
Theorem~\ref{t:primecubed}. In this case, the group $G=G_4$ has order $p^4$;
its centre and Frattini subgroup coincide, and $Z(G)=\langle a^p,b^p\rangle$,
elementary abelian of order $p^2$, while its derived subgroup is generated by
$a^p$ and is cyclic of order $p$. The calculations in the proof of
Theorem~\ref{t:primecubed} show that $(a^ib^j)^p=a^{pi}b^{pj}$. So elements
outside $Z(G)$ have order $p^2$.

Any automorphism must map $a$ to an element whose $p$th power lies in the
derived group, necessarily of the form $a^ib^{pj}$, and $b$ to an element not
of this form, say $a^kb^l$ where $p\nmid l$. Now the proof continues almost
exactly as in the proof of Theorem~\ref{t:primecubed}.

For $n>4$, we can complete the proof by induction. We note that the Frattini
subgroup of $G$ is generated by $a^p$ and $b^p$, and is abelian, with structure
$C_p\times C_{p^{n-3}}$; its Frattini subgroup $M$ is the group generated by 
$a^{p^2}$ (isomorphic to $C_{p^{n-4}}$). Thus $M$ is a nontrivial characteristic
subgroup of $G$. If $H$ is an integral of $G$, then $M$ is normal in $H$, and
$(H/M)'=H'/M=G/M\cong G_4$. But we showed above that $G_4$ is not integrable,
so this is a contradiction.\qed
\end{proof}

Now we consider product constructions, and show the following.

\begin{prop}
Let $G=G_1\times G_2$.
\begin{enumerate}\itemsep0pt
\item If $G_1$ and $G_2$ are integrable, then so is $G$.
\item If $G$ is integrable and $\gcd(|G_1|,|G_2|)=1$, then $G_1$ and $G_2$
are integrable.
\item If $G_1$ is centreless and $G_2$ is abelian, then $G$ is integrable if
and only if $G_1$ is integrable.
\end{enumerate}
\label{p:products}
\end{prop}

\begin{proof}
(a) Suppose that $H_i'=G_i$ for $i=1,2$, and let $H=H_1\times H_2$. Then
\[H'=H_1'\times H_2'=G_1\times G_2=G.\]

\smallskip

(b) Suppose that $H$ is an integral of $G_1\times G_2$.
Then $G_1$ is a characteristic subgroup of $G_1\times G_2$, and hence is
normal in $H$, and is contained in $H'$. Thus
\[(H/G_1)'=H'/G_1\cong G_2,\]
so $G_2$ is integrable; and similarly $G_1$ is integrable.

\smallskip

(c) Suppose that $G_1\times G_2$ is integrable, say $H'=G_1\times G_2$. By
assumption, $G_2=Z(G_1\times G_2)$, so $G_2$ is a characteristic subgroup of
$G$, and thus is normal in $H$. Then
\[(H/G_2)'=H'/G_2=(G_1\times G_2)/G_2\cong G_1,\]
so $G_1$ is integrable. The converse is clear.\qed
\end{proof}

The ``centreless'' condition in part (c) is essential. For example,
$D_8$ is not integrable, but $C_2\times D_8$ has an integral of order~$128$.

\section{Abelian groups}
\label{s:abel}

As we noted earlier, Guralnick observed that finite abelian groups are
integrable. Indeed, $A\wr C_2$ is an integral of the abelian group $A$,
and has order $2n^2$ if $|A|=n$.

It is possible to construct much smaller integrals of abelian groups in most
cases. If $A$ is an abelian group of odd order, then the group
\[\langle A,t\mid t^2=1,t^{-1}at=a^{-1}\hbox{ for all }a\in A\rangle\]
is an integral of $A$ of order $2|A|$. Since any finite abelian group is the
direct product of a group of odd order and a $2$-group, the results of the
last section show that it is enough to consider the latter.

Observe the following:
\begin{itemize}
\item
If $A\cong(C_{2^m})^2=\langle a_1,a_2\rangle$, then
\[\langle A,s:s^3=1,s^{-1}a_1s=a_2,s^{-1}a_2s=a_1^{-1}a_2^{-1}\rangle\]
is an integral of $A$ of the form $A\rtimes C_3$.
\item
If $A\cong(C_{2^m})^3$, there is similarly an integral of $A$ of the form
$A\rtimes C_7$. (This is a little more complicated than the previous: there
we used the integer matrix $\pmatrix{0&1\cr-1&-1\cr}$ of order $3$. There is
no $3\times3$ integer matrix of order~$7$; but there is such a matrix over the
$2$-adic integers. Equivalently, $(C_2)^3$ has an automorphism of order~$7$
(and is the derived subgroup of the semidirect product), and this automorphism
can be lifted to $(C_{2^m})^3$ for all $m$.)
\item If $A\cong C_{2^m}$, then the dihedral group of order $2^{m+2}$ is an
integral of $A$.
\end{itemize}

Thus, a finite abelian group $A$ has an integral of order at most
$42 \cdot 2^m \cdot |A|$, where $m$ is the number of powers $2^a$ for which the expression
for $A$ as a direct product of cyclic groups of prime power order has a unique
factor of order $2^a$. For this we extend each such cyclic factor to one twice
as large; then extend by a cyclic group of order $42$, where the element of
order $2$ inverts these cyclic groups and the odd-order part of $A$, while
elements of orders $3$ and $7$ act as previously described on products of two
or three cyclic $2$-groups of the same order. (Any number greater than $1$ can
be written as a sum of $2$s and $3$s.)
Noting that $|A|\ge 2^{1+2+\cdots+m}=2^{m(m+1)/2}$, we see that the
order of this integral is at most $|A|^{1+o(1)}$.

Can this bound be reduced to $c|A|$ for some constant $c$? We see that, to
answer this question, we need to consider direct products of cyclic $2$-groups
of distinct orders.

\begin{lemma}\label{lemma1}
Let $H$ be a $2$-group acting by automorphisms on the finite elementary abelian $2$-group $A$, then
\[|A/[A,H]|\ge |A| ^{1/|H|}.\]
\end{lemma}

\begin{proof} By induction on $|H|$.  Let $H = \langle x\rangle$ have order $2$. Then for every $a\in A$,
\[[a,x]^x = [a,x]^{-1} =  [a,x],\]
hence $[A,x]\le C_A(x)$. On the other hand, the map $a\mapsto [a,x]$ is a homomorphism of the abelian group $A$, and so \[ \left|\frac{A}{[A,x]}\right| \ge \left|\frac{A}{C_A(x)}\right| = |[A,x]|\]
which is what we want.

Let now $|H|\ge 4$, let $Z$ be a central subgroup of order $2$ of $G$ and $\overline A = A/[A,Z]$. Then $H/Z$ acts on $\overline A$ and, by inductive assumption,
\[|\overline A/[\overline A,H]|\ge |\overline A|^{1/|H/Z|} = |\overline A| ^{2/|H|}.\]
Now clearly $[\overline A, H] = [A,H]/[A,Z]$, whence
\[\left| \frac{A}{[A,H]}\right| = \left| \frac{\overline A}{[\overline A,H]}\right|\ge |\overline A|^{2/|H|}\ge \left(|A|^{1/2}\right)^{2/|H|} = |A| ^{1/|H|}.\]
\qed\end{proof}

\begin{lemma}\label{lemma2} 
Let $A$ be a finite elementary abelian $2$-group, and $G$ a $2$-group such that $G'=A$; writing $H = G/A$, we have
\[
|H|\log^2 |H| \ge 2\log |A|.
\]
\end{lemma}

\begin{proof} Let $T = [A,G]$. Then, by Lemma \ref{lemma1},
\[|A|^{1/|H|} \le |A/T|.\]
On the other hand, $A/T = (G/T)'$ is an elementary abelian subgroup of $Z(G/T)$; by standard arguments it follows that $G/T$ modulo its centre is elementary abelian of order, say, $2^t$; moreover
\[ |A/T| \le 2^{{t\choose 2}} \le |H|^{(t-1)/2} \le |H|^{\log|H|/2}.\]
Hence
\[ |A|^{1/|H|} \le |H|^{\log|H|/2},\]
i.e. $|H|\log^2 |H| \ge 2\log |A|$.\qed
\end{proof}

\begin{prop}\label{px1}
Let $A$ be an abelian $2$-group which is a direct product of $m$ cyclic groups
of distinct orders. Suppose that $A=G'$ for some group $G$. Then
$|G:A|\to\infty$ as $m\to\infty$.
\end{prop}

\begin{proof}
Let $A=C_{2^{a_1}}\times\cdots\times C_{2^{a_m}}$, with $a_1,\ldots,a_m$ in
strictly decreasing order.
The Frattini subgroup $\Phi(A)$ has the property that $A/\Phi(A)$ is elementary
abelian of order $2^m$, and automorphisms of $A$ of odd order act faithfully
on $A/\Phi(A)$. There can be no such non-identity automorphisms. For the
subgroups of $A$ consisting of elements of orders dividing $2^{a_i}$ are
characteristic, and their projections onto $A/\Phi(A)$ form a composition 
series for this group, whose terms are necessarily fixed by automorphisms of
odd order. So elements of odd order in $G$ centralise $A$, and we can assume
without loss of generality that $G$ is a $2$-group. Now the result follows from
Lemma~\ref{lemma2}.\qed
\end{proof}

We conclude:

\begin{theorem}
\begin{enumerate}
\item A finite abelian group $A$ has an integral of order at most $|A|^{1+o(1)}$.
\item There is no constant $c$ such that every finite abelian group $A$ has an
integral of order at most $c|A|$.\qed
\end{enumerate}
\label{t:abelian}
\end{theorem}

The arguments above can be refined to give explicit upper and lower bounds
for the order of the smallest integral of an abelian group.

We have computed the smallest integrals of abelian $2$-groups of orders up
to $64$. The results are in
Table~\ref{small_abelian}. The computations involved simply testing the
groups $H$ in the \texttt{SmallGroups} library in \textsf{GAP} to decide
whether $H'$ is isomorphic to the given group $G$.

\begin{table}[htbp]
\[\begin{array}{|c|c|c|}
\hline
\mbox{Order} & \mbox{Invariant} & \mbox{Smallest} \\
             & \mbox{factors} & \mbox{integral} \\
\hline
2 & (2) & 8 \\
\hline
4 & (4) & 16 \\
  & (2,2) & 12 \\
\hline
8 & (8) & 32 \\
  & (4,2) & 64 \\
  & (2,2,2) & 56 \\
\hline
16 & (16) & 64 \\
   & (8,2) & 128 \\
   & (4,4) & 48 \\
   & (4,2,2) & 128 \\
   & (2,2,2,2) & 48 \\
\hline
32 & (32) & 128 \\
   & (16,2) & 256 \\
   & (8,4) & 256 \\
   & (8,2,2) & 256 \\
   & (4,4,2) & 256 \\
   & (4,2,2,2) & 256 \\
   & (2,2,2,2,2) & 256 \\
\hline
\end{array}
\quad
\begin{array}{|c|c|c|}
\hline
\mbox{Order} & \mbox{Invariant} & \mbox{Smallest} \\
             & \mbox{factors} & \mbox{integral} \\
\hline
64 & (64) & 256 \\
   & (32,2) & 512 \\
   & (16,4) & 512 \\ 
   & (16,2,2) & 512 \\
   & (8,8) & 192 \\
   & (8,4,2) & 512 \\
   & (8,2,2,2) & 512 \\
   & (4,4,4) & 448 \\
   & (4,4,2,2) & 192 \\
   & (4,2,2,2,2) & 512 \\
   & (2,2,2,2,2,2) & 192 \\
\hline
\end{array}
\]
\caption{\label{small_abelian}Smallest integrals of abelian groups}
\end{table}

\section{Centreless groups}
\label{s:gs}

A group $G$ is \emph{complete} if its centre is trivial and
$\Aut(G)=\Inn(G)$. Equivalently, a group $G$ is complete if and only if, for
any group $H$ such that $G \unlhd H$, then
$H\cong G \times T$, for some group $T$.
This follows from \cite[Theorems 7.15 and 7.17]{rotman}
or \cite[Theorem 13.5.7]{robinson}. This has the following
consequence:

\begin{prop}
Let $G$ be a complete group. Then $G$ is integrable if and only if it is
perfect.
\label{p:complete}
\end{prop}

\begin{proof}
Suppose that $H'=G$. Then $G\unlhd H$, so $H\cong G\times T$ for some
$T \unlhd H$ with $G \cap T = 1$; and  $T \cong GT/G \cong H/H'$ is abelian,
and so $G'=H'=G$.

The converse is trivial since every perfect group is integrable.\qed
\end{proof}

We now turn to the more general class of centreless groups (those with
trivial centre).

If $Z(G)=1$, then $G$ is isomorphic to a subgroup of $\Aut(G)$, namely the
group of inner automorphisms of $G$. Furthermore, $\Aut(G)$ also has trivial
centre. So the process can be continued:
\[G\le\Aut(G)\le\Aut(\Aut(G))\le\cdots.\]
Wielandt's \emph{automorphism tower theorem} (for example, see
\cite[Theorem 13.5.4]{wielandt}) says that the procedure terminates after
finitely many steps.  The final group in the sequence is complete.

In this section, we show that the same is true for reduced integrals of $G$.
Let $G$ be a group with $Z(G)=1$. We say that an integral $H$ of $G$ is
\emph{reduced} if $C_H(G)=1$. Asking for a reduced
integral removes the ``constant of integration'' (abelian direct factor), but
does more than this.

For example, let $G=A_5$, and let $H$ be a semidirect product of $G$ with a
cyclic group of order $4$ whose generator induces on $G$ the automorphism of
conjugation by a transposition. Then $C_H(G)$ is cyclic of order $2$, and
$H/C_H(G)$ is an integral of $G$ isomorphic to $S_5$.

\begin{lemma}
\label{thm:reduced-integral}
If $Z(G)=1$ and $H$ is an integral of $G$, then
\begin{enumerate}\itemsep0pt
\item $C_H(G)=Z(H)$;
\item $H/C_H(G)$ is a reduced integral of $G$.
\end{enumerate}
\end{lemma}

\begin{proof}
(a) Clearly $C_H(G)\ge Z(H)$. Take $h\in C_H(G)$, so that $g^h=g$ for all
$g\in G$. Let $t$ be any other element of $H$. Then $g^t \in G=H' \unlhd H$.
Then
\[g^{[h,t]}=g^{h^{-1}t^{-1}ht} =g\]
for all $g\in G$, so $[h,t]\in C_H(G)$. But $[h,t]\in H'= G$, and
$G\cap C_H(G)=Z(G)=1$. So $[h,t]=1$ for all $t\in H$, whence $h \in Z(H)$.

(b) We have
\[(H/C_H(G))'=H'C_H(G)/C_H(G)=GC_H(G)/C_H(G)\cong G/(G\cap C_H(G))=G,\]
so $H/C_H(G)$ is an integral of $G$. But $H/C_H(G)$ is isomorphic to the
group of automorphisms of $G$ induced by $H$; so $H/C_H(G)$ acts faithfully
on $G$, whence the centraliser of $G$ is trivial.\qed
\end{proof}

Note that, if $Z(G)=1$ and $H$ is a reduced integral of $G$, then $Z(H)=1$,
so the process can be continued.

\begin{prop}
Let $G$ be a finite group with $Z(G)=1$. Suppose that
\[G=G_0<G_1<G_2<\cdots,\]
where $G_{n+1}$ is a reduced integral of $G_n$ for all $n$. Then the sequence
terminates after finitely many steps.
\label{p:upper_bound}
\end{prop}

\begin{proof}
Note that $G_0$ is normal in $G_n$ for all $n$, since it is
the $n$th term of the derived series of $G_n$.

We prove that $C_{G_n}(G_0)=1$ for all $n$. The proof is by induction
on $n$; it holds by definition for $n=1$. So
let us assume the result for $n$ and take $g\in C_{G_{n+1}}(G_0)$.

For any $h\in G_{n+1}$, we see that $[g,h]$ centralises $G_0$, and
$[g,h]\in G_n$; so $[g,h]\in C_{G_n}(G_0)$, whence $[g,h]=1$. As this is
true for all $h\in G_{n+1}$, we have $g\in Z(G_{n+1})$. But by  
Lemma \ref{thm:reduced-integral} and the construction, this forces $g=1$.

Now this means that $G_n$ is embedded in $\mathrm{Aut}(G_0)$ for all $n$,
so $|G_n|$ is bounded by a function of $G_0$, and the sequence terminates.\qed
\end{proof}

\begin{cor}
Let $G$ be a group of order $n$ with $Z(G)=1$. Then, if $G$ is integrable,
it has an integral of order at most $n^{\log_2n}$.
\end{cor}

\begin{proof}
$G$ can be generated by at most $\log_2n$ elements (using Lagrange's theorem,
since by induction on $m$ the group generated by $m$ independent elements
has order at least $2^m$). Now an automorphism is determined by its effect
on the generators, and each generator has at most $n$ possible images under
any automorphism.\qed
\end{proof}

\section{Groups with composition length~$2$}

In this section, we analyse groups with composition length at most $2$,
and give a procedure to determine whether or not such groups are integrable, up
to a specific question about outer automorphism groups of simple groups.

First we consider groups with composition length~$1$, and observe:

\begin{prop}
If $G$ is a finite simple group, then $G$ has an integral.
\end{prop}

\begin{proof}
This is easily seen by considering three cases:
\begin{itemize}\itemsep0pt
\item If $G=C_2$, then the dihedral and quaternion groups $D_8$ and $Q_8$ are
integrals of $G$.
\item If $G=C_p$, then the dihedral group of order $2p$ is an integral of $G$.
\item If $G$ is a non-abelian simple group, then it is perfect, so it is its
own integral.\qed
\end{itemize}
\label{p:simple}
\end{proof}

For groups with composition series of length $2$, we begin with a simple
observation.

\begin{lemma}
\label{thm:composition-series-perfect}
Let $G$ be a group with the property that, in every composition series
$G>N>\cdots$ for $G$, the factor group $G/N$ is a non-abelian simple group.
Then $G$ is perfect, and hence is its own integral.
\end{lemma}

\begin{proof}
For if $G$ is not perfect, then it has an abelian quotient, and hence a normal
subgroup with quotient $C_p$. Taking this as the start of a composition series
gives the result.\qed
\end{proof}

\medskip

Now let $G$ be a group with composition series $G>N>\{1\}$. By
Lemma \ref{thm:composition-series-perfect}, we
may suppose that $G/N\cong C_p$ for some prime $p$.

\subparagraph{Case 1:} $N\cong C_q$ for some prime $q$. There are two
possibilities:
\begin{itemize}\itemsep0pt
\item $G=C_p\times C_q$. Then $G$ is abelian, and so has an integral.
\item $p\mid q-1$ and $G$ is non-abelian. Then $G$ does not have an integral.
For suppose that $H$ is an integral of $G$. Since $N=G'$, we see that $N$ is
normal in $H$. The automorphism group of the cyclic group $N$ is abelian,
and so $H'$ acts trivially on $N$. This contradicts the fact that $H'=G$ and
$G$ induces a group of order $p$ of automorphisms of $N$.
\end{itemize}

\subparagraph{Case 2:} $N$ is a non-abelian simple group. Again there are two
subcases:
\begin{itemize}\itemsep0pt
\item If $C_p$ induces the trivial outer automorphism of $N$, then we can 
change the generator by an element of $N$ so that it acts trivially on $N$;
then $G=N\times C_p$, and both factors have integrals, and hence so does $G$.
\item In the other case, $C_p$ is a subgroup of $\Out(N)$. If $H$ is an
integral of $G$, then $H/N$ is an integral of $G/N\cong C_p$ inside
$\Out(N)$. 

But such a subgroup may or may not exist. (For example, if 
$\Out(N)\cong D_6$ and $p=2$, then $C_2$ has an integral, but not within
$\Out(N)$, so $G$ has no integral.) Resolving this case will require some
detailed analysis of the outer automorphism groups of simple groups.
\end{itemize}

The previous discussion allows us to recover the following folklore result (see \cite{miller1}).

\begin{prop}
For every $n \ge 5$, symmetric group $S_n$ is non-integrable.
\end{prop}

\begin{proof}
For every $n \ge 5$, the symmetric group $S_n$ has composition length $2$ with series
$S_n>A_n>\{1\}$. We recall that $\mathrm{Out}(A_6) \cong C_2 \times C_2$ and $\mathrm{Out}(A_n) \cong C_2$ for all $n\ne 6$.
Since $T:=A_n/S_n \cong C_2$ and its smallest integral $S$ is $D_8$, then $S$ is not contained in $\mathrm{Out}(A_n)$ for any $n\ge 5$,
the discussion above implies that $S_n$ is non-integrable. \qed
\end{proof}

\section{Orders of non-integrable groups}
\label{s:orders}

This section is devoted to the proof of the following theorem:

\begin{theorem}
Let $n$ be a positive integer. Then every group of order $n$ is integrable if
and only if $n$ is cube-free and there do not exist prime divisors
$p$, $q$ of $n$ with $q\mid p-1$.
\label{t:orders}
\end{theorem}

\begin{remark}
It is a formal consequence of the statement of this theorem that non-integrable
groups of all even orders greater than $4$ exist: 
if $n$ is even and $n>4$ then either $n=2^d$ for $d\ge3$, or $n$ is divisible
by an odd prime $p$ (with $2\mid p-1$).
However, we already know this because of Corollary \ref{thm:dihedral-no-integral}, and we use this result in the proof; so we assume that $n$ is odd.
\end{remark}

\begin{remark}
It is interesting to compare this theorem with the description of numbers $n$
for which all groups of order $n$ are abelian. Since abelian groups are
integrable, this set is a subset of the set in the Theorem. It is known that
all groups of order $n$ are abelian if and only if $n$ is cube-free and there
do not exist primes $p$ and $q$ such that either
\begin{itemize}\itemsep0pt
\item $p$ and $q$ divide $n$, and $q\mid p-1$; or
\item $p^2$ and $q$ divide $n$, and $q\mid p+1$.
\end{itemize}
(This result is folklore. For a proof by Robin Chapman, see~\cite{xxxx}.) So
the two sets contain no non-cube-free integers, and coincide on squarefree
integers; but there are integers (such as $75$) for which non-abelian groups
exist but all groups are integrable.
\end{remark}

\paragraph{Proof}
We show that the conditions on $n$ in the Theorem are necessary. If $n$
is not cube-free, then $n=p^am$ where $p$ is prime, $a\ge3$, and $p\nmid m$.
By Theorem~\ref{t:primecubed} and Proposition~\ref{p:primepower},
there is a non-integrable group $P$ of order $p^a$; since $\gcd(p^a,m)=1$,
then the direct product of $P$ with any group of order $m$ is not integrable,
by Proposition~\ref{p:products}(b).
If primes $p$ and $q$ divide $n$, with $q\mid p-1$, then the
non-abelian group of order $pq$ is centreless and non-integrable by
Corollary~\ref{c:pq}, so its direct product with any group of order $n/pq$ is
non-integrable, by Proposition~\ref{p:products}(c).

So suppose that $n$ satisfies these conditions. If $n$ is even, then $n=2$ or
$n=4$; then all groups or order $n$ are abelian, and so integrable (Section
\ref{s:abel}). So we may assume from now on that $n$ is odd. Assume
(for a contradiction) that there exists a non-integrable group of order $n$;
inductively, we may suppose that $n$ is minimal subject to this. 

Our strategy is to show that $G$ has a normal subgroup $N$ which is a direct
product of elementary abelian groups of order $p^2$ for various primes $p$,
and an abelian complement $H$ which normalises each of these factors of $N$,
whose action on $N$ gives each factor (of order $p^2$, say) the structure of
the additive group of $\GF(p^2)$ such that the action of $H$ on $N$ corresponds
to multiplication by a subgroup of the multiplicative group of the finite
field of order dividing $p+1$.

When this is achieved, we let $K=\langle G,t\rangle$, where $t^2=1$, $t$
normalises $N$ and acts on each factor of order $p^2$ as the field automorphism
of order $2$ of $\GF(p^2)$ (that is, as the map $x\mapsto x^p$), and on $H$ by
inversion (the map $x\mapsto x^{-1}$). A short calculation shows that this
gives an action of $t$ by automorphism on $G$. (If the order of an element $x$
of the multiplicative group of $\GF(p^2)$ divides $p+1$, then $x^p=x^{-1}$.)
Now commutators $[t,g]$ for $g\in G$ generate $G$. (If $h\in H$, then
$[h,g]=h^{-2}$; since $|H|$ is odd, these elements generate $H$. If $P$ is a
Sylow $p$-subgroup of $N$, then the commutators $[g,t]$, for $g\in P$, 
generate a non-trivial subgroup of $P$; since $H$ acts irreducibly on $P$, the
$H$-conjugates of this generate $P$.) So $K'=G$, contradicting the assumption
that $G$ is not integrable.

So it remains to prove that $G$ has the structure described above. In
particular, we have to show that $G$ is metabelian. The Odd-Order Theorem
shows that $G$ is soluble; but this can be proved much more easily using
Burnside's transfer theorem, as follows.

Since $n$ is cube-free, the Sylow $p$-subgroups of $G$ have order $p$ or
$p^2$, and so are abelian. Let $p_1<p_2<\ldots<p_r$ be the primes dividing $n$.
Let $P_1$ be a Sylow $p_1$-subgroup of $G$. Now the automorphism group of
$P_1$ has order either $p_1(p_1-1)$ or $p_1(p_1-1)^2(p_1+1)$, the fact that
$p_1$ is the smallest prime divisor of $n$ and is odd shows that $N_G(P_1)$
acts trivially on $P_1$, and so is equal to $C_G(P_1)$. By Burnside's
Transfer Theorem, $G$ has a normal $p_1$-complement $G_1$. By induction we
construct normal subgroups $G_2,\ldots,G_r=1$ so that $G_{i+1}$ is a normal
$p_i$-complement in $G_i$, and the quotient is abelian. Thus $G$ is soluble,
as claimed.

Let $F$ be the Fitting subgroup of $G$, the largest nilpotent normal subgroup
of $G$. Since $F$ is the direct product of its Sylow subgroups, and these are
abelian, $F$ is abelian. Moreover, $F$ contains its centraliser
\cite[Theorem 6.1.3]{gorenstein}, and so $C_G(F)=F$. Thus, $G/F$ is isomorphic
to the group of automorphisms of $F$ induced by conjugation in $G$. This
group is a subdirect product of the groups induced on the Sylow subgroups
of $F$.

Now the automorphism group of $C_p$ is $C_{p-1}$; the automorphism group of
$C_{p^2}$ is $C_{p(p-1)}$; and the automorphism group of $C_p\times C_p$ is
$\GL(2,p)$, of order $p(p-1)^2(p+1)$. Now if $p$ divides $n$, then $p-1$ is
coprime to $n$ by assumption; and if $F$ contains a Sylow $p$-subgroup then
$p$ does not divide $|G/F|$. Thus the group induced on a Sylow $p$-subgroup
$P$ of $F$ is trivial if $P$ is cyclic, and has order dividing $p+1$ if
$P$ is elementary abelian of order $p^2$. Moreover, from the structure of
$\GL(2,p)$, we see that a subgroup of order dividing $p+1$ is cyclic, and
corresponds to multiplication by an element in the multiplicative group of
$\GF(p^2)$ acting on the additive group of this field.

So $G/F$ is a subdirect product of cyclic groups, and hence is abelian.
Moreover, cyclic Sylow subgroups of $F$ are central in $G$.

Now we define $N$ to be the product of the Sylow subgroups of $F$ which are
elementary abelian of prime squared order. We see that $G/N$ has a subgroup
$F/N$ (generated modulo $N$ by the cyclic Sylow subgroups of $F$) with 
$(G/N)/(F/N)\cong G/F$ abelian. Thus $G/N$ is an extension of a central
subgroup by an abelian group, and so it is nilpotent of class at most~$2$.
But then it is a direct product of its Sylow subgroups, and so is abelian.

We also note that, if $G/N$ acts trivially on a Sylow $p$-subgroup $P$ of
$N$, then by Burnside's transfer theorem, $G\cong P\times G_1$ for a subgroup
$G_1$ of order $n/p^2$; by the minimality of $n$, we have that $G_1$ is
integrable, and has order coprime to $p$, so that $G$ is integrable, a
contradiction. So the action of $G/N$ on each Sylow subgroup is non-trivial,
and each such subgroup has the structure of a finite field $\GF(p^2)$, with
the induced automorphism group isomorphic to a subgroup of the multiplicative
group, acting irreducibly.

Finally, $N$ is a normal Hall subgroup of $G$; if we take $H$ to be a Hall
subgroup for the complementary set of primes, then $H$ is a complement for
$N$ in $G$, and $H\cong G/N$, so $H$ is abelian, and we have reached our goal.
\qed

\section{Miscellanea}
\label{s:misc}

\subsection{Products, subgroups, quotients}
\label{ssec:psq}

We saw in Proposition~\ref{p:products}(a) that, if $G_1$ and $G_2$ have
integrals, then so does $G_1\times G_2$. But what about the converse? In other
words, is it possible that $G_1\times G_2$ has an integral but $G_1$ and $G_2$
do not? We saw earlier in Proposition~\ref{p:products}(b) that this is not
possible if the orders of $G_1$ and $G_2$ are coprime.

A particular case of the above question is:
can $G\times G$ be integrable when $G$ is not integrable?
We do not have an example. The smallest non-integrable group is $S_3$; and
$S_3\times S_3$ is also non-integrable. (For this group has trivial centre,
so if it were integrable it would have a reduced integral, which would be
contained in the automorphism group of $S_3\times S_3$; but this automorphism
group has order $72$, and its derived group has order $18$.)

The next case is $D_8\times D_8$. We have shown
that it has no integral of order at most $512$.

A related question would replace direct product by central product. However,
this question has a negative answer. It is well-known that $D_8\circ D_8$ is
isomorphic to $Q_8\circ Q_8$, which has the integral
$\mathrm{SL}(2,3)\circ\mathrm{SL}(2,3)$.

Regarding semidirect products, there are groups $H$ and $G$, with  $H$ acting on $G$ into two different ways, say $\phi$ and $\psi$, such that the semidirect product induced by $\phi$ is integrable but the semidirect product induced by $\psi$ is not. For one such example take $G=C_4$ and  $H=C_2$, and the two possible actions of $H$ on $G$ (the trivial action, and the action by inversion).
There are also groups $G$ and $H$ such that every semidirect product they can form fails to be integrable: take the dihedral group of order $10$ and the
Klein four group.  Finally, {\em all} semidirect products of two copies of
$C_2$ are integrable. So all possibilities involving integrals of semidirect
products can occur.

\medskip

Regarding subgroups and homomorphic images, note that the group $G=A_5$
(which is perfect, so integrable) contains $H$, the dihedral group of order
$10$, which is not integrable (by Proposition 5.1), but has a normal subgroup
$K$, the cyclic group of order $5$, such that $H$ is the normalizer of $K$ and
both $K$ and $H/K$ are
integrable. So neither integrability nor non-integrability is subgroup-closed,
and a group can have the property that all its non-trivial proper normal
subgroups are integrable with integrable quotient without itself being
integrable. Moreover, any finite group has both integrable and non-integrable
overgroups (since the alternating group is integrable but the symmetric group
is usually not).

In the reverse direction, we have the following:

\begin{prop}
Let $G$ be an integrable finite group. Then either $G$ is simple, or $G$ has
a non-trivial proper quotient which is integrable.
\end{prop}

\begin{proof}
Let $H$ be an integral of $G$. There are two cases:
\begin{itemize}\itemsep0pt
\item[] \textbf{Case 1:} $G$ has a non-trivial proper characteristic subgroup
$N$ (one invariant under all automorphisms of $G$). Then $N$ is normal in $H$,
and we have $(H/N)'=H'/N=G/N$.
\item[] \textbf{Case 2:} $G$ is characteristically simple. In this case, $G$
is a direct product of isomorphic simple groups. So either $G$ is simple, or
it has a simple (and hence integrable) proper quotient.\qed
\end{itemize}
\end{proof}

\subsection{Relative integrals}
\label{sec:rel-int}

Let $U$ be a ``universal'' group. Can we decide, for members $G$ of some class
of subgroups of $U$, whether or not $G$ has an integral within $U$?

This question includes several special cases which have arisen elsewhere in
this paper:

\begin{itemize}\itemsep0pt
\item $U=\mathrm{Aut}(T)$ for some non-abelian simple group $T$, and $G$ is
a subgroup containing $T$.
\item $U$ is the affine group $\mathrm{AGL}(d,p)$, and $G$ is a subgroup
containing the translation group.
\item $U$ is the symmetric group $S_n$, and $G$ is a (transitive, or maybe
$2$-homogeneous) permutation group of degree $n$.
\end{itemize}

The third problem for $2$-transitive groups involves solving some special
cases of the other two, since a $2$-transitive group is either affine or
almost simple.

More generally, we could ask for a classification of 
the primitive groups $H\le S_n$ that have a given $2$-homogeneous group
$G\le S_n$ in their integrals tower. (Equivalently, find the primitive
subgroups of $S_n$ that appear as the derived group of the derived group of
the \ldots  of a given $2$-homogeneous group $G\le S_n$.)   
This question should not be difficult; examining the known list of
$2$-homogeneous groups should give a solution, since in all cases the derived
length of the soluble residual is very small.

\subsection{Infinitely integrable groups}
\label{ss:infint}

There are groups which can be integrated $n$ times (that is, which are
isomorphic to the $n$th derived group of another group), for all $n$. For
example, additive groups of rings with identity have this property: 
if $G$ is the group of upper unitriangular $(2^n+1)\times(2^n+1)$ matrices
over the ring $R$ with identity, then the $n$th derived group of $G$ consists
of unitriangular matrices with zeros in all above-diagonal positions except the
top right, and is isomorphic to the additive group $R^+$ of $R$. (We note that,
since an upper unitriangular matrix has determinant $1$, Cramer's rule shows
that its inverse can be found by ring operations alone, and so any product of
commutators in this group over any ring with identity can be computed with
ring operations.) 

We have the following result about groups with this property.

\begin{theorem}
\begin{enumerate}
\item Let $G$ be a finite abelian group; then $G$ can be integrated $n$ times for every natural number $n$ (even within the class of finite nilpotent groups). 
\item A finite group can be integrated $n$ times for every natural number $n$ if and only if it is central product of a perfect group and an abelian group.
\end{enumerate}
\end{theorem}

\begin{proof} (a) Since any finite abelian group is a direct product of finite cyclic groups, it is enough to show the result for these. Now the cyclic group $C_n$ is the additive group of the ring
$R_n=\mathbb{Z}/n\mathbb{Z}$, and so is the $n$th derived group of the group
of upper unitriangular matrices of order $2^n+1$ over $R_n$. Since the group
of upper unitriangular matrices is nilpotent, the result is proved.

Observe that, in this way, one may realize any finite abelian group $A$ as the $n$th derived group of a nilpotent group, in which $A$ is central.

\smallskip

(b) Let $G = NA$,  with $N, A$ normal subgroups, $N$ perfect, $A$ abelian, and $A\cap N =  Z(N)$, and let $n$ be a natural number. By point (a) there  exists a group $R$ such that $A = R^{(n)}$ and $A\le Z(R)$. Now, if $H$ is the central product $H = X\circ R$, then $H^{(n)} = G$.

Conversely, let $G$ be a finite group. Then there is  $k$ such that $X^{(k)}$ is perfect for every $X\le \Aut (G)$. Let $N$ be the soluble residual of $G$ (the smallest normal subgroup such that $G/N$ is soluble); then $N$ is characteristic in $G$. Suppose there exists a group $H$ such that $G = H^{(n)}$, for some $n \ge k$, and let $C = C_H(G)$. Then $NC/C$ is the soluble residual of $H/C$ and $NC \ge H^{(k)} \ge G$. Hence, $A=G\cap C\le Z(G)$ and   $G = NC\cap G = N(C\cap G) = NA$, thus proving the claim.\qed
\end{proof}

Combining this result with Theorem~\ref{t:orders}, we obtain the following
result:

\begin{cor}
For a natural number $n$, the following are equivalent:
\begin{enumerate}\itemsep0pt
\item every group of order $n$ is abelian;
\item every group of order $n$ can be integrated twice;
\item every group of order $n$ can be integrated $k$ times, for every natural
number $k$;
\item $n$ is cubefree and has no prime divisors $p$ and $q$ such that either
$q\mid p-1$, or $q\mid p+1$ and $p^2\mid n$.
\end{enumerate}
\end{cor}

\begin{proof}
We saw in Section~\ref{s:orders} the (classical) equivalence of (a) and (d).
We have observed that (a) implies (c), and trivially (c) implies (b). So
suppose that $n$ is such that every group of order $n$ can be integrated
twice. Then by Theorem~\ref{t:orders}, $n$ is cubefree and has no prime
divisors $p$ and $q$ with $q\mid p-1$. So suppose that $p$ and $q$ are
primes, with $q\mid p+1$ and $p^2\mid n$. Let $H$ be the group of order
$p^2q$ which is a semidirect product of the additive group $N$ of the field of
order $p^2$ with a subgroup of index $q$ in the multiplicative group. Then
as argued there, the only reduced integral of $G$ is obtained by adjoining
the Frobenius automorphism $t$ of the field. If $\langle G,t\rangle$ is
integrable, then so is $\langle G,t\rangle/N$. But this group is dihedral
of order $2q$, with $q$ odd, and so is not integrable.\qed
\end{proof}

One can ask if there are infinitely integrable groups. More precisely,
even if a group $G$ is not solvable, one can ask whether there exists
an infinite chain of finite groups of the form
\[
G=G_1' \le G_1=G_2' \le G_2=G_3' \le \ldots 
\]

The answer to this question is positive if we allow perfect groups, by taking
$G_i=G$ for all $i$. However, 
Neumann~\cite[Corollary 7.5]{neumann1} showed that there is no strictly 
ascending infinite sequence if $G_2$ is finitely generated; so in particular,
there is no such sequence of finite groups.
(Note that Neumann~\cite{neumann1} gives an example of an infinite ascending
sequence where $G_0$ and $G_1$ are finite and the other terms infinite.)

If we relax the conditions slightly, the following construction gives examples
of groups $G$ with subgroups $G_n$ for all natural numbers $n$, such that 
$G_0$ is finite (but $G_n$ infinite for $n>0$), and other examples where all
$G_n$ are finite but the second condition is weakened to the pair of conditions
\begin{itemize}\itemsep0pt
\item $G_n'\ge G_{n-1}$ for $n>0$,
\item $G_n^{(n)}=G_0$.
\end{itemize}

\paragraph{Construction:}

Let $R$ be a ring with identity. (We are particularly interested in the case
where $R$ is finite; but the construction works in general.) Let $I$ be the
set of dyadic rational numbers in $[0,1]$, and $I_n$ the subset of $I$ in which
the denominators are at most $2^n$. (So $I_0=\{0,1\}$.)

Our groups will be contained in the group of upper triangular matrices over 
$R$, where the index set of rows and columns is $I$. Let $e_{ij}$ be a
symbol for each $i,j\in I$ with $i<j$. Our group will be generated by elements
$xe_{ij}$ for all such $i,j$ and all $x\in R$; the relations are
\begin{itemize}\itemsep0pt
\item $xe_{ij}\cdot ye_{ij}=(x+y)e_{ij}$;
\item for $i<j<k$, $[xe_{ij},ye_{jk}]=xye_{ik}$;
\item if $j\ne k$ and $i\ne l$, then $[xe_{ij},ye_{kl}]=1$.
\end{itemize}

Think of $xe_{ij}$ as the matrix $I+xE_{ij}$, where $E_{ij}$ has entry $1$ in
the $(i,j)$ position and $0$ elsewhere.

Let $G_n$ be the subgroup generated by $xe_{ij}$ for $i,j\in I_n$, and
$H_d$ be generated by $x_{ij}$ with $j-i\ge d$. The groups $G_n$ are
isomorphic to the group of $(2^n+1)\times(2^n+1)$ strictly upper triangular
matrices over $R$, and so are finite if $R$ is finite. The groups $H_d$
are infinite if $d<1$; $G_0=H_1$ is isomorphic to the additive group of $R$.

Moreover, we see that
\begin{eqnarray*}
&& G_n'=G_n\cap H_{1/2^{n-1}}\ge G_{n-1};\\
&& H_d'=H_{2d},
\end{eqnarray*}
with the convention that $H_d=\{1\}$ if $d>1$.

Hence the chains $(H_{1/2^n})$ and $(G_n)$ of groups have the properties
claimed.

\section{Infinite groups}
\label{s:infinite}

We have rather less to say about infinite groups.

The definition of integral applies equally to finite and infinite groups.
Several of our results (Lemma~\ref{l:const}, Proposition~\ref{p:cc},
Proposition~\ref{p:products}(a) and (c), and Proposition~\ref{p:complete})
also apply to infinite groups.

Following the first part of the proof of Theorem~\ref{t:finite}, we show:

\begin{prop}
Let $G$ be finitely generated. If $G$ has an integral, then it has a finitely
generated integral.
\end{prop}

\begin{proof}
Suppose that $H$ is an integral of $G$. Take a finite generating set for $G$,
and write each generator as a product of commutators of elements of $H$.
Then the finite set of elements involved in these commutators generate a
subgroup of $H$ whose derived group is $G$.\qed
\end{proof}

We mention a couple of classes of infinite groups which are integrable.

\begin{theorem}
\begin{enumerate}\itemsep0pt
\item Any abelian group is integrable.
\item Any free group is integrable.
\end{enumerate}
\end{theorem}

\begin{proof}
(a) We follow Guralnick's proof~\cite{gural1}: if $A$ is abelian, then the
derived group of $A\wr C_2$ is the subgroup $\{(a,a^{-1}):a\in A\}$ of the
base group $A^2$ of the wreath product, which is clearly isomorphic to $A$.

\smallskip

(b) If $\alpha$ is an infinite cardinal, then $|F_\alpha|=\alpha$, so
$|F_\alpha'|\le\alpha$. The inequality cannot be strict, since if 
$\{f_i:i\in\alpha\}$ is a free generating set for $F_\alpha$, then the elements
$[f_0,f_i]$, for $i\in\alpha$, $i>0$, of the derived group are all distinct.
So $F_\alpha'\cong F_\alpha$.

For finite rank, we use the result of Nielsen~\cite{nielsen}, a special case
of which asserts that the derived group of the free product $C_{m_1}*C_{m_2}$
is a free group of rank $(m_1-1)(m_2-1)$. So the derived group of $C_2*C_{n+1}$
is $F_n$.\qed
\end{proof}

\begin{remark}
The free product of integrable groups may or may not be integrable. For example,
\begin{itemize}\itemsep0pt
\item $C_2*C_2$ is the infinite dihedral group, which is not integrable (by the
same argument as for finite dihedral groups).
\item $C_3*C_3$ is the derived group of $\mathrm{PGL}(2,\mathbb{Z})$
\cite{jonesetal,sahinetal}.
\end{itemize}
\end{remark}

In view of part (a) above and our results on abelian groups, we could ask
whether any abelian group $A$ has an integral $G$ with $|G:A|$ finite. For
example, if $A=A^2$ (that is, every element of $G$ is a square), then the
\emph{generalized dihedral group}
\[G=\langle A,t\rangle\mid t^2=1, t^{-1}at=a^{-1}\hbox{ for all }a\in A\rangle\]
is an integral of $A$ with $|G:A|=2$. But in general the answer is negative:

\begin{theorem}
There exist infinite abelian groups $A$ having no integral $G$ with 
$|G:A|$ finite.
\end{theorem}

\begin{proof}
Following the arguments given in Section~\ref{s:abel}, we take $A$ to be the
direct product of cyclic groups of orders $2^k$ for $k\in\mathbb{N}$.

Let $G$ be a $2$-group such that $G'=A$, and suppose for a contradiction that
$|G/A|$ is finite. Then we can find subgroups $N\le A$ with $N\unlhd G$ and
$A/N$ finite and arbitrarily large. But $A/N=(G/N)'$, contradicting Proposition
\ref{px1}.\qed
\end{proof}

In fact much more can be said about integrals of infinite abelian groups;
this will be discussed in a later paper.

\section{Open problems}

We conclude with a list of problems arising from this study, some of which have
already been mentioned.

\begin{problem}
Let $N$ be the set of positive integers $n$ for which every group
of order $n$ is integrable (Theorem~\ref{t:orders}).
Does $N$ have a density? What is its density? (We note that of the integers
up to $10^8$, the number which lie in $N$ is $32261534$.)
\end{problem}

\begin{problem}
Find a good upper bound for the order of the smallest integral of an integrable
group of order $n$.
\end{problem}

\begin{problem}
True or false? For a fixed prime $p$, the proportion of groups of order $p^n$
which are integrable tends to $0$ as $n\to\infty$.
\end{problem}

\begin{problem}
\begin{enumerate}
\item[(a)]Let us call a group $G$ \emph{almost integrable} if $G\times A$ is integrable
for some abelian group $A$. We saw in Proposition~\ref{p:products}(c) that,
for centreless groups, ``integrable'' and ``almost integrable'' are equivalent.
Which groups with non-trivial centre are almost integrable? (At the end of
Section \ref{s:exs}, we noted that $C_2 \times D_8$ is integrable, so $D_8$
is almost integrable, but not integrable.)
\item[(b)] Does there exist a finite non-integrable group $G$, such that $G\times G$ is integrable? In particular, is
$D_8\times D_8$ integrable? (As noted in Subsection \ref{ssec:psq} it has no integral of order less
than $512$.) 
\end{enumerate}
\end{problem}

\begin{problem}
For the three cases mentioned in Subsection \ref{sec:rel-int} decide,
for members $G$ of some class of subgroups of $U$, whether or not $G$ has an integral within $U$. In the context of the discussion following Lemma \ref{thm:composition-series-perfect}, especially Case 2, we are particularly interested in the case where $U=\mathrm{Out}(T)$ for some simple group $T$ and $G$ is cyclic.
\end{problem}

\begin{problem}
For which finite non-abelian groups $G$ is it true that, for all finite groups
$H$ with $G'=H'$, it holds that $H$ is integrable if and only if $G$ is? (All
finite abelian groups have this property, but it fails for $D_8$ and $Q_8$.)
\end{problem}

\begin{problem}
One difficult problem is relating integrability of $G$ to that
of $G/Z(G)$, or indeed $G/Z$ for any central subgroup $Z$ of $G$. Is there
a cohomological tool that would help? One direction is trivial: if $G$ has
an integral then so does $G/Z(G)$.

Let us suppose that $Z(G)\cong C_p$; put $H=G/Z(G)$. Then $G$ is ``described''
by a cohomology class $\gamma$ in the second cohomology group (or extension
group) $H^2(H,C_p)$. Suppose that $H$ has an integral $K$. Does cohomology
provide a tool to decide whether $\gamma$ is the restriction to $H$ of a
class $\delta\in H^2(K,C_p)$? If so, then the corresponding extension is
an integral of $G$.
\end{problem}

\begin{problem}
Which integrable infinite groups $G$ have an integral $H$ such that the index
of $G$ in $H$ is finite?
\end{problem}

\begin{problem}
let $G$ be a locally finite group which is \emph{locally integrable} (that is,
every finite subset is contained in an integrable subgroup). Must $G$ be
integrable? If so, must there be a locally finite integral?
\end{problem}

%

\begin{problem}
Let $X$ be a $n$-set and $S_n$ the symmetric group on $X$. Does there exist an equivalence relation $\rho$ on $X$ such that the group $G\le S_n$ of all permutations that preserve $\rho$ is integrable in $S_n$? We know this is false for the identity or universal relations; is it false for all equivalence relations?

Such a group $G$ is a direct product of wreath products $S_a\wr S_b$ of
symmetric groups. As a preliminary step, we could ask whether a direct
product of symmetric groups can be integrable.
\end{problem}

\begin{problem}
Determine whether the following problem is undecidable: given a presentation $\langle X \mid R\rangle$ for a group $G$, is $G$ integrable? Are there decidable instances of this problem? For example, is the problem decidable for one-relator groups?
\end{problem}

\begin{problem}
It is known that the infinite finitely presented Thompson groups $T$ and $V$ are simple. On the
other hand, Thompson's group $F$ has simple commutator subgroup, but is not itself simple. Is the group $F$ integrable?
\end{problem}

\begin{problem}
\begin{enumerate}\itemsep0pt
\item
The class of all integrals of a given variety $V$ of groups is a variety of groups $W$. Given a base of identities for $V$, is it possible to find a base of identities for $W$?
\item Let $G$ be a finite integrable group, and $W$ the variety of integrals
of groups in $V=\mathrm{Var}(G)$. Is there an integral $H$ of $G$ such that
$W=\mathrm{Var}(H)$?
\end{enumerate}
\end{problem}

Related to the previous problem we have the following. 

\begin{problem}	
Is it possible to classify the finite sets  $A\subseteq F_2$, the $2$-generated free group, such that the group $$\langle a,b\mid w(a,b)=1=w(b,a) \ (w\in A)\rangle$$ satisfies $w(x,y)=1$, for all $w\in A$?  
\end{problem}

\begin{problem}
Every group in a variety of abelian groups has integral. Are there other varieties with this property?
\end{problem}

\begin{problem}
For any integrable group there is a smallest integral; how many different integrals of smallest order can there be? 
\end{problem}

\begin{problem}
Is it true that no Coxeter group with connected diagram, apart from $C_2$, is integrable?
\end{problem}

\begin{problem}
Produce some algorithms and effective GAP code to find integrals of a given group. 
\end{problem}

\begin{problem}
As there is a classification of the groups in which all subgroups are normal,  classify the groups in which all subgroups are integrable.
\end{problem}

\begin{problem}
It makes sense to adapt the integrability concept to Lie algebras via the derived subalgebra.
Is it true that
a Lie algebra is integrable if and only if the corresponding Lie group is?
\end{problem}

\begin{problem}
Is there a ring theoretic analogue of the results in this paper, now taking $[a,b]=ab-ba$?
\end{problem}

Observe that in general, the commutators of all pairs of elements in a ring form a subring, but they do not necessarily form an ideal \cite{mcleod}. Therefore the ring theory literature considers the commutator as the ideal generated by all pairs $[a,b]$. Nevertheless, they form a right ideal if and only if they form a left ideal: $(ab-ba)c = a(bc-cb) + (ac)b-b(ac)$.

\section*{Acknowledgments}
The first, second and fourth authors gratefully acknowledge the support of the 
Funda\c{c}\~ao para a Ci\^encia e a Tecnologia  (CEMAT-Ci\^encias FCT project UID/Multi/04621/2013).

The third and fourth authors are members of the Gruppo Nazionale per le Strutture Algebriche, Geometriche e le loro Applicazioni (GNSAGA) of the Istituto Nazionale di Alta Matematica (INdAM). The fourth author gratefully acknowledges the support of the 
Funda\c{c}\~ao de Amparo \`a Pesquisa do Estado de S\~ao Paulo 
(FAPESP Jovens Pesquisadores em Centros Emergentes grant 2016/12196-5),
of the Conselho Nacional de Desenvolvimento Cient\'ifico e Tecnol\'ogico (CNPq 
Bolsa de Produtividade em Pesquisa PQ-2 grant 306614/2016-2).

\end{document}